\documentclass[reqno]{amsart}
\usepackage{amsthm,amsfonts,amssymb,mathrsfs}
\usepackage[usenames]{color}
\newtheorem{theorem}{Theorem}[section]
\newtheorem{lemma}[theorem]{Lemma}
\newtheorem{assumption}[theorem]{Assumption}

\theoremstyle{remark}
\newtheorem{remark}[theorem]{Remark}
\newtheorem{definition}[theorem]{Definition}

\numberwithin{equation}{section}

\renewcommand{\bar}[1]{\overline{#1}}
\renewcommand{\tilde}[1]{\widetilde{#1}}

\newcommand{\MM}{\mathcal{M}}

\newcommand{\RR}{\mathcal{R}}

\newcommand{\C}{\mathbb{C}}
\newcommand{\R}{\mathbb{R}}
\newcommand{\N}{\mathbb{N}}

\newcommand{\al}{\alpha}
\newcommand{\be}{\beta}
\newcommand{\ga}{\gamma}
\newcommand{\de}{\delta}

\renewcommand{\th}{\theta}

\newcommand{\la}{\lambda}

\newcommand{\Ga}{\Gamma}
\newcommand{\De}{\Delta}

\newcommand{\Si}{\Sigma}

\newcommand{\Om}{\Omega}

\renewcommand{\Im}{\mathop{\textrm{\upshape{Im}}}}

\newcommand{\diag}{\text{diag}}

\newcommand{\ldr}{\left|\!\left|\!\left|}
\newcommand{\norm}{|\!|\!|}
\newcommand{\rdr}{\right|\!\right|\!\right|}

\newcommand{\neu}[1]{\color{blue}{#1}\color{black}{}}
\renewcommand{\texttt}[1]{\neu{#1}}
\newcommand{\weg}[1]{\color{red}{#1}\color{black}{}}
\renewcommand{\textsf}[1]{\weg{#1}}
\newcommand{\id}{\mathop{\textrm{id}}\nolimits}

\renewcommand{\neu}[1]{#1}
\renewcommand{\weg}[1]{}
\newcommand{\sm}{\setminus}

\allowdisplaybreaks
\begin{document}
%%%%%%%%%%%%%%%%%%%%%%%%%%%%%%%%%%%%%%%%%%%%%%%%%%%%%%%%%%%%%%%%%%%%%%%%%%%%%%%%%%%%%%%%%%%%%%%%%%%%%%%%%%%%%%%%%%%%%%%%%%%%%%%%%
\title[Estimates for a  transmission problem of mixed type]{$L^p$-estimates for a  transmission problem of mixed elliptic-parabolic type}
\author{Robert Denk}
\author{Tim Seger}
\address{University of Konstanz, Department of Mathematics and Statistics, 78457 Konstanz,
Germany}
\email{robert.denk@uni-konstanz.de, tim.seger@uni-konstanz.de}
\keywords{Transmission problem, elliptic-parabolic equation, a priori estimates}
 \subjclass[2010]{35B45, 35M12}
 
\begin{abstract} We consider the situation when an elliptic problem in a subdomain $\Omega_1$ of an $n$-dimensional bounded domain $\Omega$ is coupled via inhomogeneous canonical transmission conditions to a parabolic  problem in $\Omega\setminus\Omega_1$. In particular, we can treat elliptic-parabolic equations in bounded domains with discontinuous coefficients.  Using Fourier multiplier techniques, we prove an a priori estimate for strong solutions to the equations in $L^p$-Sobolev spaces.\end{abstract}
\date{July 25, 2013}
\maketitle
\rmfamily

%\bigskip
%\fbox{\color{blue}Blau = ge\"andert, \color{red}rot = wird gestrichen.\color{black}}

%%%%%%%%%%%%%%%%%%%%%%%%%%%%%%%%%%%%%%%%%%%%%%%%%%%%%%%%%%%%%%%%%%%%%%%%%%%%%%%%%%%%%%%%%%%%%%%%%%%%%%%%%%%%%%%%%%%%%%%%%%%%%%%%%%%%%%
\section{Introduction}
In the present paper we  prove a priori estimates in $L^p$-Sobolev spaces for the solution of a transmission problem of elliptic-parabolic type with discontinuous coefficients. More precisely, we consider a bounded domain $\Omega\subset\R^n$ which is divided into two subdomains $\Omega_1,\Omega_2$ separated by a closed contour $\Gamma\subset\Omega$ and a boundary value problem of the form
\begin{equation}
  \label{1-1}
  \begin{aligned}
    A(x,D) u & = f_1\quad \text{ in }\Omega_1,\\
    (A(x,D)-\lambda) u & = f_2\quad\text{ in }\Omega_2,\\
   C(x,D) u & = h\quad \text{ on }\partial\Omega.
  \end{aligned}
  \end{equation}
  Here $A(x,D)$ is a differential operator of order $2m$,  $C(x,D)$ is a column of boundary operators $C_1,\dots,C_m$, and $\lambda$ is a complex parameter.
We assume  $f_k\in L^p(\Omega_k)$ and are looking for a solution $u\in W_p^{2m}(\Omega)$. The top-order coefficients of the operator $A(x,D)$ are assumed to be continuous up to the boundary in each subdomain ${\Omega_k}$ but may have jumps across the interface $\Gamma$. The condition $u\in W_p^{2m}(\Omega)$ leads to the canonical transmission conditions along $\Gamma$, given by
\begin{equation}
  \label{1-2}
   [\![ \partial_\nu^{j-1} u]\!] = 0\quad (j=1,\dots,2m),
\end{equation}
where $[\![ \partial_\nu^{j-1} u]\!] $ stands for the jump of the $(j-1)$-th normal derivative of $u$ along the interface $\Gamma$. Generalizing \eqref{1-2}, we will consider inhomogeneous transmission conditions of the form
\begin{equation}
  \label{1-3}
  B(x,D) u = g,
\end{equation}
where $B=(B_1,\dots,B_{2m})^\top$ with
\[ B_j(x,D) u := \partial_\nu^{j-1} u_1 - \partial_\nu^{j-1} u_2\quad (j=1,\dots,2m).\]
Here we have set $u_k := u|_{\Omega_k}$ for $k=1,2$.

The aim of the paper is to prove uniform a priori estimates for the solutions of \eqref{1-1}, \eqref{1-3} under suitable ellipticity and smoothness assumptions on $A$ and $C$, see Section~2 below for the precise formulations. To give an idea of our results, let us for the moment assume that $f_1=f_2=0$ and $h=0$ in \eqref{1-1}, \eqref{1-3}. In  classical elliptic theory, in the case of an uncoupled system we would expect  a uniform a priori estimate of the form
\[ \|u_1\|_{W_p^{2m}(\Omega_1)} \le C \Big( \sum_{j=1}^{m} \| g_j\|_{W_p^{2m-j+1-1/p}(\Gamma)} + \|u_1\|_{L^p(\Omega_1)}.\Big)\]
On the other hand, the classical parabolic (in the sense of parameter-elliptic) a priori estimate would read as
\[ \norm u_2\norm_{W_p^{2m}(\Omega_2)} \le C \sum_{j=1}^{m} \norm g_j\norm_{W_p^{2m-j+1-1/p}(\Gamma)}.\]
Here $\norm\cdot\norm_{W_p^s} := \|\cdot\|_{W_p^s} + |\lambda|^{s/2m} \|\cdot\|_{L^p}$ is the typical parameter-dependent norm appearing in parabolic theory. Concerning the coupled system \eqref{1-1}, \eqref{1-3}, the question arises if we still have similar estimates for $u_1$ and $u_2$. We will see below that this is true in some sense. More precisely, we will obtain
\begin{align*}
  &\|u_1\|_{W_p^{2m}(\Omega_1)} + \|u_2\|_{W_p^m(\Omega_2)}  \le C \Big( \sum_{j=1}^{2m} \|g_j\|_{W_p^{2m-j+1-1/p}(\Gamma)} + \|u_1\|_{L^p(\Omega_1)}\Big),\\
  & |\lambda|^{1/2} \|u_1\|_{W_p^{m}(\Omega_1)} + \norm u_2\norm_{W_p^m(\Omega_2)}   \le C \Big( \sum_{j=1}^{2m} \norm g_j\norm_{W_p^{2m-j+1-1/p}(\Gamma)} \\
  &\hspace*{20em}+ |\lambda|^{1/2} \|u_1\|_{L^p(\Omega_1)}\Big).
\end{align*}
This can be seen as a mixture of elliptic and parabolic a priori estimates. Note that we do not reach the full order $2m$ with respect to $u_2$ in the first inequality and not the full power $|\lambda|$ with respect to $u_1$ in the second inequality. The general result for $f\not=0$ and $h\not=0$ and the precise formulation are stated in Section~2 below.

Applications of problem \eqref{1-1}, \eqref{1-3} (in its parabolic form, i.e., the parameter $\lambda$ being replaced by the time derivative)  can be found, e.g., in \cite{Geb07}, including the heat equation in a domain with vanishing thermal capacity in some subdomain and a model of an electric field generated by a current in a partially non-conducting domain.   On the other hand, the  problem under consideration is closely related to  spectral problems with indefinite weight functions of the form
\[ (A(x,D)-\lambda\omega(x)) u = f\;\text{ in }\Omega,\quad C(x,D)u = 0\;\text{ on }\partial\Omega.\]
Here $\omega$ is a weight function which may change sign and may vanish on a set of positive measure. Such spectral problems have been investigated, e.g., in a series of papers by Faierman (see \cite{Fai00}--\cite{Fai09}) and by Pyatkov (\cite{Pya98}, \cite{PA02}), see also \cite{Beh12} and the references therein. In particular, in the paper \cite{Fai09} a Calder\'{o}n method of reduction to the boundary was applied to deal with the case where $\omega$ vanishes on a set $\Omega_1$ of positive measure. For this, unique solvability of the Dirichlet boundary value problem in $\Omega_1$ had to be assumed. Transmission problems of purely parabolic type (where the parameter $\lambda$ is present in each subdomain) and $L^p$-a priori estimates for their solution were considered in \cite{ADF97}. Transmission problems in $L^p$ were also studied, with the same methods as in the present note, by Shibata and Shimizu in \cite{SS11}.

A standard approach to \neu{treat  transmission} problems is to use (locally) a reflection technique in one subdomain resulting in a \textit{system} of differential operators which are coupled by the transmission conditions. A general theory of parameter-dependent systems can be found in a series of papers by Volevich and his co-authors (see  \cite{DMV00} and the references therein). Here the so-called Newton polygon method leads to uniform a priori estimates for the solution. However, in the present case the Newton polygon is of trapezoidal form and thus not regular. Therefore, the Newton polygon approach cannot be applied to the transmission problem \eqref{1-1}. On the other side, the resulting system is not \texttt{parameter-elliptic} in the classical sense (\cite{AGRVISH})  and is not covered by the standard \texttt{parameter-elliptic} theory. We also note the connection to singularly perturbed problems where a similar Newton polygon structure appears, cf. \cite{DV00}. The analysis of the elliptic-parabolic system \texttt{below}  also serves as a starting point for more general (and nonlinear) elliptic-parabolic systems as, for instance, appearing  in lithium battery models (see \cite{WXZ06}). A detailed investigation of the nonlinear elliptic-parabolic lithium battery model and solvability in $L^p$-Sobolev spaces can be found in the second author's thesis \cite{Seg13}. In \cite{Latz} and \cite{Dreyer} mathematical models for lithium battery systems can be found which lead to inhomogeneous transmission conditions.

In Section~2 we will state the precise \texttt{assumptions} and the main \texttt{result} of the present paper. The boundary value problem is analyzed by a localization method and the investigation of the model problem in the half-space. An explicit description of the solution of the model problem (in terms of Fourier multipliers) and resulting estimates can be found in Section~3. Finally, the proof of the main a priori estimate is given in Section~4.
%%%%%%%%%%%%%%%%%%%%%%%%%%%%%%%%%%%%%%%%%%%%%%%%%%%%%%%%%%%%%%%%%%%%%%%%%%%%%%%%%%%%%%%%%%%%%%%%%%%%%%%%%%%%%%%%%%%%%%%%%%%%%
\section{Statement of the problem and main result}
Let $1<p<\infty$, $n\in\N$, $k\in\N_0:=\{0,1,2,\dots\}$, and $\Om \subset \R^n$ be open. By $\left(\texttt{L^p}(\Om), \|\cdot\|_{0,p,\Om}\right)$ and $\left(W^k_p(\Om), \|\cdot\|_{k,p,\Om}\right)$ we denote the Lebesgue and Sobolev spaces on $\Om$ with their standard norms. We will further make use of the seminorms $$|u|_{k,p,\Om}:=\sum_{|\al|=k} \|D^{\al}u\|_{0,p,\Om}\quad (k\in \texttt{\N_0}, u\in W^k_p(\Om)),$$
 where we used the standard notation $D^{\al}:= (-i)^{|\al|}(\frac{\partial}{\partial x_1})^{\alpha_1}\ldots (\frac{\partial}{\partial x_n})^{\alpha_n}$.
For real \texttt{non-integer} $s>0$ let $W^s_p(\Om):=B^s_{pp}(\Om)$ denote the Besov space on $\Om$ with its standard norm. Besides the standard norms, for the treatment of parameter-elliptic problems the following parameter-dependent norms will be convenient: Let \texttt{$\theta\in (0,\pi]$ and} let $\la \in \overline{\Si}_{\texttt{\theta}}$ be a complex parameter, varying in a closed sector \texttt{$\overline{\Si}_{\texttt{\theta}}$ with vertex at $0$ where $\Sigma_\theta :=\{z\in \C\setminus\{0\}:\ |\arg(z)| <\theta\}$}. Then for $m\in\N$ and $k\in\{0,1,\dots,2m\}$, we \texttt{define}
\begin{equation}
\ldr u \rdr_{k,p,\Om}:= \|u\|_{k,p,\Om} + |\la|^{\frac{k}{2m}} \|u\|_{0,p,\Om} \quad(u\in W^k_p(\Om)).
\label{paramdepnorm}
\end{equation}
On the boundary, we will consider parameter-dependent trace norms given by
\[ \ldr u\rdr_{k-1/p,p,\Gamma} := \|u\|_{k-1/p,p,\Gamma} + |\la|^{\frac{k-1/p}{2m}} \|u\|_{0,p,\Gamma}\quad (u\in W_p^{k-1/p}(\Gamma)).\]
By
$\mathcal Fu$ we denote the Fourier transform of $u$, and $(\mathcal F'u)(\xi',x_n)$ stands for
 the partial Fourier transform  with respect to the first $n-1$ variables $x':=(x_1,\ldots,x_{n-1})$.
%%%%%%%%%%%%%%%%%%%%%%%%%%%%%%%%%%%%%%%%%%%%%%%%%%%%%%%%%%%%%%%%%%%%%%%%%%%%%%%%%%%%%%%%%%%%%%%%%%%%%%%%%%%%%%%%%%%%%%%%%%%%%%%%%%%
%%%%%%%%%%%%%%%%%%%%%%%%%%%%%%%%%%%%%%%%%%%%%%%%%%%%%%%%%%%%%%%%%%%%%%%%%%%%%%%%%%%%%%%%%%%%%%%%%%%%%%%%%%%%%%%%%%%%%%%%%%%%%%%%%%%

Let $\Om\subset \R^n$ be a bounded domain with boundary $\partial\Om$ of class $C^{2m-1,1}$, and let $\Ga$ be a closed \texttt{$C^{2m-1,1}$} Jordan contour in $\Om$, having no points with $\partial \Om$ in common. Denote by $\Om_1$ and $\Om_2$ the \texttt{resulting} subdomains such that $\Om_1\cap \Om_2=\emptyset$, \texttt{$\partial\Omega_1=\Gamma$,} and $\overline{\Om}= \overline{\Om_1\cup \Om_2}$.  Note that, due to our assumptions, there is no contact point of $\Gamma$ and $\partial\Omega$. We define $u_i:=\left.u\right|_{\Om_i}$ and will consider the differential operators $A_1(x,D)=A(x,D)$ for $x\in \Om_1$ and $\texttt{\tilde A_2}(x,D)-\la= A(x,D)-\la$ for $x\in \Om_2$. Slightly generalizing the form of equation \eqref{1-1}, we \texttt{consider} differential operators of even order $2m$ of the following structure
$$A_1(x,D)= \sum_{|\al|\leq 2m} a^{(1)}_{\al}(x)D^{\al}\quad\text{and}\quad \texttt{\tilde{A}}_2(x,D,\la)= \sum_{|\al|+k\leq 2m} a^{(2)}_{\al k}(x)\la^{k/2m}D^{\al}$$
with $m\in \N$ and $\la\in \overline{\Si}_{\th}$ for some $\th\in[0,\pi)$. Furthermore, let \texttt{the} boundary operators $C_j$ of order $0\leq m_j \leq 2m-1$  be of the form
$$C_j(x,D)= \sum_{|\ga|\leq m_j} c_{j\ga}(x)D^{\ga},$$
being defined on $\partial\Om$. We will write for short $(A,C_1,\ldots,C_m)$ when we refer to the boundary value problem \eqref{1-1}.
%%%%%%%%%%%%%%%%%%%%%%%%%%%%%%%%%%%%%%%%%%%%%%%%%%%%%%%%%%%%%%%%%%%%%%%%%%%%%%%%%%%%%%%%%%%%%%%%%%%%%%%%%%%%%%%%%%%%%%%%%%%%%%%%%%
\begin{assumption}\label{assump1}
\begin{itemize}
	\item[(1)] Smoothness assumptions on the coefficients. We assume
$$a^{(1)}_{\alpha} \in \begin{cases}C(\overline{\Om_1})&\ (|\al|=2m),\\ L^{\infty}(\Om_1)& \ (|\al|< 2m),\end{cases}\qquad  a^{(2)}_{\alpha k} \in \begin{cases}C(\overline{\Om_2})&\ (|\al|=2m),\\ L^{\infty}(\Om_2)& \ (|\al|< 2m)\end{cases}$$
for the coefficients of the differential operators and $c_{j\ga} \in C^{2m-m_j-1,1}(\partial\Om)$ for the coefficients of the boundary operators.
	\item[(2a)] Ellipticity of $A_1$. For the principal symbol $A_1^{0}(x,\xi):= \sum_{|\al|= 2m} a^{(1)}_{\al}(x)\xi^{\al}$, we have $A_1^{0}(x,\xi) \neq 0 \quad (x\in\overline{\Om}_1,\ \xi\in\R^n\sm\{0\})$.
	\item[(2b)] Ellipticity with parameter of the boundary value problem $(\texttt{\tilde{A}}_2,C_1,\ldots,C_m)$. The principal symbol of $\tilde{A}_2$ satisfies
\[
\texttt{\tilde{A}}_2^{0}(x,\xi,\la): = \sum_{|\al|+k=2m} a^{(2)}_{\al k}(x) \la^{k/2m}\xi^{\al} \neq 0\]
\texttt{for all} $x\in\overline{\Om}_2$ \texttt{and all} $(\xi,\la)\in(\R^n\times \bar\Si_{\th})\sm \{(0,0)\}$,
	and the Shapiro-Lopatinskii condition is satisfied for $(\texttt{\tilde{A}_2},C_1,\ldots,C_m)$ at each point $x_0\in \partial \Om$. If $C^{0}_{j}(x,D)= \sum_{|\ga|=m_j} \texttt{c_{j\gamma}(x)D^\gamma}$ denotes the principal symbol of the boundary operator, this condition reads as follows: For $x_0\in \partial\Om$ let the boundary value problem $(\texttt{\tilde A}_2,C_1,\ldots,C_m)$ be rewritten in local coordinates associated with $x_0$, i.e. in coordinates resulting from the original ones by rotation and translation such that the positive $x_n$-axis coincides with the direction of the inner normal vector. Then for all $(\xi',\la)\in(\R^{n-1}\times \bar\Si_{\th})\sm \{(0,0)\}$ and $h_j\in \C$, the ODE problem on the halfline
\begin{align*}
\texttt{\tilde A}_2^{0}(x_0,\xi',D_n,\la) v(x_n) &= 0\quad\hspace{0.45em} \text{in }(0,\infty),\\
C_j^0(x_0,\xi',D_n) v(x_n) &= h_j\quad \text{at}\ x_n=0,\quad (j=1,\ldots,m),\\
v(x_n) &\to 0\quad\hspace{0.35em} (x_n\to \infty)
\end{align*}
admits a unique solution. Here, $D_n := -i\frac\partial{\partial x_n}$.
	\item[(3)] Assumptions on the data. We \texttt{assume}  $f_1\in \texttt{L^p}(\Om_1)$, $f_2\in \texttt{L^p}(\Om_2)$, $g_j \in W_p^{2m-j+1-1/p}(\Ga)$ for $j=1,\ldots,2m$, and $h_j \in W_p^{2m-m_j-1/p}(\partial\Om)$ for $j=1,\ldots,m$.
	\item[(4)] In addition, we assume proper ellipticity, i.e. the polynomials $A_1^0(x,\xi',t)$ and $\texttt{\tilde A}^0_2(x,\xi',t,\la) \in \C[t]$  of order $2m$ from conditions $(2a)$ and $(2b)$ have exactly $m$ roots in each half-plane $\C_{\pm}:=\{z\in\C: \pm \Im z>0\}$ for all $x\in \bar\Om_1$ and $x\in\bar\Om_2$, respectively, and for all  $\xi'\in\R^{n-1}\sm \{0\}$ and $\la\in\bar\Si_{\th}$. Proper ellipticity allows a decomposition of the form $\texttt{A_1^0}(x,\xi',t) = A_{1+}(x,\xi',t) A_{1-}(x,\xi',t)$ with
\begin{equation}\label{decomposition}
A_{1+}(x,\xi',t):= \prod_{j=1}^m (t-\tau_j(x,\xi')) \; \text{and} \; A_{1-}(x,\xi',t):= \prod_{j=m+1}^{2m} (t-\tau_j(x,\xi')),
\end{equation} where $\tau_j$ denote the roots in $\C_+\ (j\leq m)$ and $\C_-\ (j>m)$, respectively. A similar decomposition with an additional dependence on $\la$ also holds for $\texttt{\tilde A_2^0}$. We remark that proper ellipticity holds automatically if $n\ge 3$.
\end{itemize}
\end{assumption}
%%%%%%%%%%%%%%%%%%%%%%%%%%%%%%%%%%%%%%%%%%%%%%%%%%%%%%%%%%%%%%%%%%%%%%%%%%%%%%%%%%%%%%%%%%%%%%%%%%%%%%%%%%%%%%%%%%%%%%%%%%%%%%%
Under these  assumptions, we consider the inhomogeneous transmission boundary value problem
\begin{equation}\label{TP}
\begin{aligned}
A_1(x,D)\ u_1 &= f_1\quad \text{in } \Om_1,\\
\tilde A_2(x,D,\lambda)\ u_2 &= f_2\quad  \text{in } \Om_2,\\
B_j(x,D) u & =  g_j\quad\hspace{0.1em} \text{on }\Ga\quad(j=1,\ldots,2m),\\
C_j(x,D)u_2 &= h_j\quad\text{on } \partial\Om\quad (j=1,\ldots,m).
\end{aligned}
\end{equation}
Here we have set $B_j(x,D) u :=\partial_\nu^{j-1} u_1 - \partial_\nu^{j-1} u_2$ where $\partial_\nu$ denotes the derivative in direction of the outer normal with respect to $\Omega_2$.
Our main result is the following a priori estimate for solutions to \eqref{TP}. Here, a solution of \eqref{TP} is defined as a pair $(u_1,u_2)$ belonging to the Sobolev space $ W_p^{2m}(\Omega_1)\times W_p^{2m}(\Omega_2)$ for which the system \eqref{TP} is satisfied as equality of  $L^p$-functions.
%%%%%%%%%%%%%%%%%%%%%%%%%%%%%%%%%%%%%%%%%%%%%%%%%%%%%%%%%%%%%%%%%%%%%%%%%%%%%%%%%%%%%%%
\begin{theorem}[A priori estimate for the transmission boundary value problem]\label{mainresult_domain}
Let Assumption \ref{assump1} be satisfied and let $u = (u_1,u_2) \in W_p^{2m}(\Om_1)\times W_p^{2m}(\Om_2)$ be a solution to the transmission problem \eqref{TP}. Then there exists $\lambda_0>0$ such that for all $\lambda\in\bar\Sigma_\theta$ with $|\lambda|\ge \lambda_0$ the following estimates hold:
\begin{align}
\|u_1\|_{2m,p,\Om_1} &+ \|u_2\|_{m,p,\Omega_2} + |\lambda|^{1/2} \|u_2\|_{0,p,\Om_2}  \nonumber\\
&  \leq C \biggl( \|f_1\|_{0,p,\Om_1}+ \|f_2\|_{0,p,\Om_2}+ \|u_1\|_{0,p,\Om_1} \nonumber \\
& \qquad +  \sum_{j=1}^{2m}\| g_j\|_{2m-j+1-1/p,p,\Ga} + \sum_{j=1}^{m}\ldr h_j\rdr_{2m-m_j-1/p,p,\partial\Om}\biggr),\label{1-4}\\
\|u_1\|_{2m,p,\Om_1} &+ |\lambda|^{1/2} \|u_1\|_{m,p,\Om_1} + \ldr u_2\rdr_{2m,p,\Om_2} \nonumber\\
&  \leq C \biggl(|\lambda|^{1/2} \|f_1\|_{0,p,\Om_1}
+ \|f_2\|_{0,p,\Om_2} + |\lambda|^{1/2}\|u_1\|_{0,p,\Om_1}\nonumber \\
 & \qquad +  \sum_{j=1}^{2m}\ldr g_j\rdr_{2m-j+1-1/p,p,\Ga} + \sum_{j=1}^{m}\ldr h_j\rdr_{2m-m_j-1/p,p,\partial\Om}\biggr).\label{aprioriabsch_domain}
\end{align}
\end{theorem}

Note that with respect to $g$, inequality \eqref{1-4} is of elliptic type and \eqref{aprioriabsch_domain} is of parameter-elliptic type. Due to the fact that the boundary operators $C_j$ act on $u_2$, we have parameter-elliptic norms with respect to $h_j$ in both inequalities.

%%%%%%%%%%%%%%%%%%%%%%%%%%%%%%%%%%%%%%%%%%%%%%%%%%%%%%%%%%%%%%%%%%%%%%%%%%%%%%%%%%%%%%%%%%%%%%%%%%%%%%%%%%%%%%%%%%%%%%%%%%
\begin{remark}\label{2.3}
Our main task will be to study the problem for constant coefficient operators $A_1(D)$ and $\texttt{\tilde A}_2(D,\lambda)$ in the half-spaces $\R^n_{\pm}$ without lower order terms. This  \texttt{simplification} can be justified by performing a localization procedure, using a finite covering $\overline{\Om}\subset \bigcup_{k=1}^N U_k$ with appropriate open sets $U_k$, a corresponding partition of unity and perturbation results. For a detailed explanation of the localization procedure, we refer to \cite{ADF97}, pp.~151--153, but here we briefly list the types of local problems one has to deal with. If $\overline{U_k}\subset \Om_i$, one faces a local elliptic ($i=1$) or \texttt{parameter-elliptic} $(i=2)$ operator in the whole space. For these situations, the estimates for $u_i$ are well-known results, see \texttt{\cite{ITFSDO}, Theorem~5.3.2,} for \texttt{the} elliptic and \cite{ADF97}, Proposition 2.5, for the \texttt{parameter-elliptic} case. If $U_k\cap \partial\Om\neq \emptyset$, the local problem is a standard boundary value problem in the \texttt{half-space} and the desired estimate is contained in \cite{ADF97}, Proposition 2.6. It remains to consider the case \texttt{where} $U_k$ intersects both $\Om_1$ and $\Om_2$, and in the sequel we restrict our considerations to the corresponding local model problem. \texttt{This} reads as
\begin{equation}\label{TP_HR}
\begin{aligned}
A_1(D)\ u_1&= f_1\quad \text{in } \R^n_+,\\
\texttt{\tilde A}_2(D,\lambda)\ \texttt{\tilde u_2} &= \texttt{\tilde f_2}\quad  \text{in } \R^n_-,\\
D_n^{j-1}\left(u_1 - u_2\right) &= g_j\quad\hspace{0.1em} \text{on }\R^{n-1},\quad(j=1,\ldots,2m).
\end{aligned}
\end{equation}
The reflection $\tau_n: \R^n\to \R^n,\ x\mapsto (x',-x_n)$ will be useful to treat problem \eqref{TP_HR}. Therefore, we will use the notation $A_2(\xi,\lambda):= \texttt{\tilde A}_2(\tau_n\texttt{(\xi)},\lambda) = \texttt{\tilde A}_2(\xi',-\xi_n,\lambda)$ for the symbol of the reflected operator, which is \texttt{parameter-elliptic} in $\R^n_+$. \texttt{We set $u_2(x):= \tilde u_2(\tau_n(x))$ and $f_2(x):=\tilde f_2(\tau_n(x))$.}

By this substitution, we may rewrite \eqref{TP_HR} as \texttt{a} system in the half-space $\R^n_+$:
\begin{equation}
  \label{TPSYS}
  \begin{aligned}
    A(D,\lambda) u & = f\quad \text{ in }\R^n_+,\\
    D_n^{j-1}\left(u_1 + (-1)^j u_2\right) & = g_j\quad(j=1,\ldots,2m) \text{ on }\R^{n-1}.
  \end{aligned}
\end{equation}
Here we have set
\[ A(D,\lambda) := \begin{pmatrix}
  A_1(D) & 0 \\ 0 & A_2(D,\lambda)
\end{pmatrix},\; u:= \binom{u_1}{u_2},\; f:= \binom{f_1}{f_2}.\]
\end{remark}
%%%%%%%%%%%%%%%%%%%%%%%%%%%%%%%%%%%%%%%%%%%%%%%%%%%%%%%%%%%%%%%%%%%%%%%%%%%%%%%%%%%%%%%%%%%%%%%%%%%%%%%%%%%%%%%%%%%%%%%%%
\begin{remark}\label{problem}
We see that the determinant of the principal symbol $\det(A^0(\xi,\lambda))=\det(A(\xi,\lambda))=A_1(\xi)A_2(\xi,\lambda)$ vanishes at the points $(0,\lambda)\in \R^{n}\times (0,\infty)$. Hence the standard theory for parameter-elliptic systems is not applicable in this case. Due to continuity and homogeneity of the principal symbols we have the estimate \begin{equation}\label{nellipticity} |A_1(\xi)A_2(\xi,\lambda)|\geq C |\xi|^{2m}(|\lambda|+|\xi|^{2m})\end{equation} with a constant $C>0$. Operators whose principal symbols allow an estimate of the form \eqref{nellipticity} are also called N-elliptic with parameter. Here the `N' stands for the Newton polygon which is related to the principal symbol. In case of \eqref{nellipticity}, the Newton polygon is not regular, and therefore this equation is not covered by the results on N-ellipticity as in \cite{DMV00}.
\end{remark}
%%%%%%%%%%%%%%%%%%%%%%%%%%%%%%%%%%%%%%%%%%%%%%%%%%%%%%%%%%%%%%%%%%%%%%%%%%%%%%%%%%%%%%%%%%%%%%%%%%%%%%%%%%%%%%%%%%%%%%%%%%
\begin{remark}
The boundary conditions in \eqref{TP_HR} are called canonical transmission conditions. In the case $g_j=0$, they are equivalent to the condition $U\in W_p^{2m}(\R^n)$ for  $$U(x',x_n):=\left\{\begin{aligned} u_1(x',x_n) \quad (x_n\geq 0),\\ \tilde u_2(x',x_n) \quad (x_n<0).\end{aligned}\right. $$
Note that in \eqref{TP_HR}  the number of conditions equals the order of the operator, in contrast to boundary value problems. We will show in Lemma~\ref{absolutelyelliptic} below that the ODE system corresponding to the transmission problem \eqref{TP_HR} is uniquely solvable. This is an analogue of the Dirichlet boundary conditions which are absolutely elliptic, i.e., for every properly elliptic operator the Dirichlet boundary value problem satisfies the Shapiro-Lopatinskii condition.
\end{remark}

%%%%%%%%%%%%%%%%%%%%%%%%%%%%%%%%%%%%%%%%%%%%%%%%%%%%%%%%%%%%%%%%%%%%%%%%%%%%%%%%%%%%%%%%%%%%%%%%%%%%%%%%%%%%%%%%%%%%%%%%%%
%%%%%%%%%%%%%%%%%%%%%%%%%%%%%%%%%%%%%%%%%%%%%%%%%%%%%%%%%%%%%%%%%%%%%%%%%%%%%%%%%%%%%%%%%%%%%%%%%%%%%%%%%%%%%%%%%%%%%%%%%%
\section{Fundamental solutions and solution operators}

To represent the solution in terms of  fundamental solutions, we start with the observation that the ODE system obtained from \eqref{TP_HR}  by partial Fourier transform is uniquely solvable. This is the analogue of the Shapiro-Lopatinskii condition for transmission problems.
For detailed discussions of this condition for boundary value problems, we refer to \cite{DHP}, Section 6.2 and \cite{WLOK}, Chapter 11.  The assertion of the following lemma is formulated for our situation of one elliptic and one \texttt{parameter-elliptic} operator but of course it also holds in the cases when both operators are of the same type.

To simplify our notation, we define  $q:= \la^{1/2m}$ and consider the differential operator $\tilde A_2(D,q) = \sum_{\texttt{|\al|+k \le 2m}} a^{(2)}_{\al k} q^{k}D^{\al} $ with $q\in\overline\Sigma:= \overline\Si_{\theta/(2m)}$.

%%%%%%%%%%%%%%%%%%%%%%%%%%%%%%%%%%%%%%%%%%%%%%%%%%%%%%%%%%%%%%%%%%%%%%%%%%%%%%%%%%%%%%%%%%%%%%%%%%%%%%%%%%%%%%%%%%%%%%%%%%
\begin{lemma}\label{absolutelyelliptic}
Suppose the operators $A_1(x,D)$ and $\texttt{\tilde{A}_2}(x,D,q)$ are elliptic and para\-meter-elliptic in $\bar\Sigma$, respectively. Fix $x_0 \in \partial \Om,\ \xi'\in\R^{n-1}\sm\{0\},\ q\in \overline{\Sigma}$ and let $h_j \in \C \; (j=1,\ldots,2m)$. Then the ODE problem
\begin{equation}\label{GDGLTP}
\begin{aligned}
A^0_1(x_0,\xi',D_n)u_1 &= 0\quad \quad(x_n >0)\\
\texttt{\tilde{A}^0_2}(x_0,\xi',D_n,q)\texttt{\tilde u_2} &= 0\quad \quad(x_n <0)\\
\left.D_n^{j-1}(u_1 - \texttt{\tilde u_2})\right|_{x_n = 0} &= h_j \quad\hspace{0.55em} (j = 1,\ldots,2m)\\
\hspace{5em}u_1(x_n) &\rightarrow  0  \quad\hspace{0.85em} (x_n\rightarrow \infty) \\
\hspace{5em}\texttt{\tilde u_2}(x_n) &\rightarrow  0  \quad\hspace{0.85em} (x_n\rightarrow -\infty)
\end{aligned}
\end{equation}
admits a unique solution.
\end{lemma}
%%%%%%%%%%%%%%%%%%%%%%%%%%%%%%%%%%%%%%%%%%%%%%%%%%%%%%%%%%%%%%%%%%%%%%%%%%%%%%%%%%%%%%%%%%%%%%%%%%%%%%%%%%%%%%%%%%%%%%%%%%
\begin{proof}
In the sequel, we do not write down the dependence of the polynomials and their roots on $x_0$ explicitly and fix $\xi'\in\R^{n-1}\sm\{0\}$ as well as $q\in \overline{\Sigma}$. We decompose $A^0_1(\xi',t)$ and $\texttt{\tilde{A}^0}_2(\xi',t,q)$ as indicated in \eqref{decomposition} into $A^0_{i\pm}(\xi',t)$. Let $\MM_1$ denote the \texttt{$m$-dimensional} space of stable solutions to
$$ A_1^0(\xi',D_n) v = 0 \quad (x_n>0), \quad v\rightarrow 0 \quad (x_n\rightarrow \infty)$$ and let $\MM_2$ denote the \texttt{$m$-dimensional} space of stable solutions to
$$ \texttt{\tilde{A}_2}^0(\xi',D_n,q) w = 0 \quad (x_n<0), \quad w\rightarrow 0 \quad (x_n\rightarrow -\infty).$$
Let $B_1:=\{v_1,\ldots,v_m\}$ and $B_2:=\{w_1,\ldots,w_m\}$ be a basis of $\MM_1$ and $\MM_2$, respectively. Then $B:=B_1\cup B_2$ is obviously a subset of the \texttt{$2m$-dimensional} space of  \weg{stable} solutions to the equation
\begin{equation}\label{produktgleichung}
\texttt{P}(\xi',D_n,q) u(x_n) := A^0_{1+}(\xi',D_n)\mathcal{A}^0_{2-}(\xi',D_n,q) u(x_n) = 0 \quad \text{on}\ \R
\end{equation}
and $B$ is linearly independent: Suppose there are nontrivial $\al_j, \be_j \in\C\; (j=1,\ldots,m) $ with $$\weg{u:= }\sum_{j=1}^m \al_j v_j =  \sum_{j=1}^m \be_j w_j.$$
Then \eqref{produktgleichung} would possess a solution which is bounded on the entire real line, which contradicts the fact that the polynomial $\texttt{P}(\xi',t,q)$ has only roots with nonzero imaginary part. Hence $B$ is a fundamental system to \eqref{produktgleichung} and the determinant of the Wronskian matrix $W(x_n)$ is nonzero:
\begin{equation}\label{wronski}
\det W(x_n) =\det\left(
\begin{array}[h]{ccc}
v_1(x_n)&\cdots & w_m(x_n)\\ \vdots& &\vdots\\ D_n^{2m-1}v_1(x_n)&\cdots & D_n^{2m-1}w_m(x_n)	
\end{array}
\right)\neq 0 \quad (x_n\in\R).
\end{equation}
Now suppose that $(v,w)$ is a solution to \eqref{GDGLTP}. Then there exist constants $\al_i, \be_i \in \C$ for $i=1,\ldots,m$, such that
\begin{equation*} v(x_n) = \sum_{j=1}^m \al_j v_j(x_n),\quad \text{and}\quad w(x_n) = -\sum_{j=1}^m \be_j w_j(x_n).\end{equation*} If we plug in this approach into the transmission conditions, we obtain the system of linear equations to determine $\al_j$ and $\be_j$: \begin{equation*}\left(
\begin{array}[h]{ccc}
v_1(0)&\cdots & w_m(0)\\ \vdots & &\vdots\\ D_n^{2m-1}v_1(0)&\cdots & D_n^{2m-1}w_m(0)	
\end{array}
\right)\left(
\begin{array}[h]{c}
\al_1\\ \vdots\\ \be_{m}	
\end{array}
\right) = \left(
\begin{array}[h]{c}
h_1\\ \vdots\\ h_{2m}	
\end{array}
\right).\end{equation*}
From \eqref{wronski} it now follows that the coefficients exist and are uniquely determined, which proves the assertion.
\end{proof}

From now on, we restrict ourselves to the model problem \eqref{TPSYS} which is the only non-standard step in the proof of the main theorem, see Remark~\ref{2.3}. We first consider the case $f=0$ in \eqref{TPSYS}, i.e. we study
\begin{equation}
  \label{eq3-1}
  \begin{aligned}
    A(D,q) u & = 0\quad \text{ in }\R^n_+,\\
    B(D_n) u & = g\quad\text{ on }\R^{n-1}.
  \end{aligned}
\end{equation}
Here $u=(u_1,u_2)^\top$, $g=(g_1,\dots,g_{2m})^\top$,
\[ A(D,q)=\begin{pmatrix} A_1(D) & 0 \\ 0 & A_2(D,q)\end{pmatrix},\quad
B(D_n) = \begin{pmatrix} B^{(1,1)}(D_n) & B^{(1,2)}(D_n)\\ B^{(2,1)}(D_n) & B^{(2,2)}(D_n)\end{pmatrix}\]
with
\begin{align*}
  B^{(1,1)}(D_n) & := \big( D_n^{j-1}\big)_{j=1,\dots,m},& B^{(1,2)}(D_n) &:= \big( (-1)^{j}D_n^{j-1}\big)_{j=1,\dots,m},\\
  B^{(2,1)}(D_n) & := \big( D_n^{j-1}\big)_{j=m+1,\dots,2m},& B^{(2,2)}(D_n) &:= \big( (-1)^{j}D_n^{j-1}\big)_{j=m+1,\dots,2m}.
\end{align*}
Note that $B^{(1,1)}(D_n)$ and $B^{(2,1)}(D_n)$ are also called generalized Dirichlet and  Neumann conditions, respectively.

Due to Lemma~\ref{absolutelyelliptic}, the ODE system corresponding to \eqref{eq3-1} is uniquely solvable. The main step in the proof of Theorem~\ref{mainresult_domain} will be to find a priori estimates for the fundamental solutions of this ODE system. In the following, $I_k$ stands for the ($k\times k$)-dimensional unit matrix.

\begin{definition}\label{def-fundamental-solution}
The fundamental solution
\[ \omega\colon (\R^{n-1}\setminus\{0\})\times (0,\infty)\times\bar\Sigma\to \C^{2\times 2m},\, (\xi',x_n,q)\mapsto \omega(\xi',x_n,q)\]
 is defined as the unique solution of the ODE system (in $x_n$)
\begin{align*}
  A(\xi',D_n,q) \omega(\xi',x_n,q) & = 0 \quad (x_n>0),\\
  B(D_n) \omega(\xi',x_n,q)\big|_{x_n=0} & = I_{2m},\\
  \omega(\xi',x_n,q) & \to 0 \quad (x_n\to\infty).
\end{align*}
\end{definition}

Following an idea of Leonid Volevich \cite{vol}, we represent the solutions in a specific way. For this, we consider the elliptic boundary value problem $(A_1(D), B^{(1,1)}(D_n))$ and the parameter-elliptic boundary value problem $(A_2(D,q), B^{(2,2)}(D_n))$ separately. It is well known that the (generalized) Dirichlet and Neumann boundary conditions are absolutely elliptic,  hence the Shapiro-Lopatinskii condition holds for both  subproblems. We will call the canonical basis for these boundary value problems the basic solutions $Y^{(1)}$ and $Y^{(2)}$. More precisely, we define:

\begin{definition}
  \label{basicfunctions}
  We define the basic solution
  \[ {Y}^{(1)}\colon (\R^{n-1}\setminus\{0\})\times (0,\infty)\to \C^{1\times m}, \, (\xi',x_n)\to {Y}^{(1)}(\xi',x_n),\]
  as the unique solution of the ODE system
  \begin{equation}\label{eq3-2}
  \begin{aligned}
    A_1(\xi',D_n) {Y}^{(1)}(\xi',x_n) & = 0 \quad (x_n>0),\\
    B^{(1,1)}(D_n) {Y}^{(1)}(\xi',x_n)\big|_{x_n=0} & = I_m,\\
    {Y}^{(1)}(\xi',x_n) & \to 0\quad (x_n\to\infty).
  \end{aligned}
  \end{equation}
  Analogously, the basic solution
  \[ {Y}^{(2)}\colon (\R^{n-1}\setminus\{0\})\times (0,\infty)\times\bar\Sigma\to \C^{1\times m},\,
  (\xi',x_n,q)\mapsto {Y}^{(2)}(\xi',x_n,q),\]
  is defined as the unique solution of the ODE system
   \begin{equation}\label{eq3-3}
  \begin{aligned}
    A_2(\xi',D_n,q) {Y}^{(2)}(\xi',x_n,q) & = 0 \quad (x_n>0),\\
    B^{(2,2)}(D_n) {Y}^{(2)}(\xi',x_n,q)\big|_{x_n=0} & = I_m,\\
    {Y}^{(2)}(\xi',x_n,q) & \to 0\quad (x_n\to\infty).
    \end{aligned}
  \end{equation}
  We set
  \[ {Y}(\xi',x_n,q) = \Big( {Y}_{k}^{(j)}(\xi',x_n,q)\Big)_{\substack{j=1,2\hfill \\k=1,\dots,2m}} := \begin{pmatrix} {Y}^{(1)}(\xi',x_n) & 0 \\ 0 & {Y}^{(2)}(\xi',x_n,q)\end{pmatrix}.\]
\end{definition}

The advantage of the basic solutions ${Y}^{(1)}, {Y}^{(2)}$ lies in the fact that classical (parame\-ter-)elliptic estimates are easily available for them. We have to compare these solutions with the fundamental solution $\omega$. Let $j\in\{1,\dots,2m\}$. As the function $\omega_{1j}$ is a solution of $A_1(\xi',D_n)\omega_{1j}=0\;(x_n>0)$, it can be written as a linear combination of the basic solutions. Therefore, we can write
\[ \omega_{1j}(\xi',x_n,q) = \sum_{k=1}^m {Y}^{(1)}_k(\xi',x_n,q) \psi_{kj}(\xi',q)\]
with unknown coefficients $\psi_{kj}$. The analogous representation holds for  $\omega_{2j}$. In matrix notation, we obtain
\[ \omega(\xi',x_n,q) = {Y}(\xi',x_n,q)\Psi(\xi',q) \]
with $\Psi(\xi',q) = \big(\psi_{kj}(\xi',q)\big)_{k,j=1,\dots,2m}$. By the definition of the fundamental solution, we have
\[ I_{2m} = B(D_n)\omega(\xi',x_n,q)\big|_{x_n=0} = B(D_n){Y}(\xi',x_n,q)\big|_{x_n=0} \Psi(\xi',q).\]
Therefore,
\begin{equation}
\label{eq3-4}
\Psi(\xi',q) \\
= \begin{pmatrix}
  I_m & \hspace*{-1em} B^{(1,2)}(D_n){Y}^{(2)}(\xi',x_n,q)\big|_{x_n=0} \\
  B^{(2,1)}(D_n){Y}^{(1)}(\xi',x_n)\big|_{x_n=0} & I_m
\end{pmatrix}^{-1}.
\end{equation}

%%%%%%%%%%%%%%%%%%%%%%%%%%%%%%%%%%%%%%%%%%%%%%%%%%%%%%%%%%%%%%%%%%%%%%%%%%%%%%%%%%%%%%%%%%%%%%%%%%%%%%%%%%%%%%%%%%%%%%%%%%%%%%%%%%%%
\begin{remark}\label{coeffproperty}
Due to the unique solvability of the equations \eqref{eq3-2} and \eqref{eq3-3}, we have for $(\xi',q)\in (\R^{n-1}\setminus\{0\})\times \overline{\Sigma}$ the following scaling properties for all $r>0$:
\begin{eqnarray*}
{Y}^{(1)}\left(\tfrac{\xi'}{r},r x_n\right) & = & {Y}^{(1)}(\xi',x_n) \Delta_1(r) ,\\
{Y}^{(2)}\left(\tfrac{\xi'}{r},r x_n,\tfrac{q}{r}\right) &=&  {Y}^{(2)}(\xi',x_n,q) \Delta_2(r),
\end{eqnarray*}
where we used the abbreviations
\begin{align*}
 \De_1(r)&:= \diag(1,r,\ldots,r^{m-1}),\\
 \De_2(r)& := \diag(r^m,\ldots,r^{2m-1}) = r^m \De_1(r).
 \end{align*}
 (See also \eqref{eq3-8} below for an explicit representation of $Y^{(1)}$ and $Y^{(2)}$.)
We will apply this with $r:=|\xi'|$ for ${Y}^{(1)}$ and $r:= |\xi'|+|q|$ for ${Y}^{(2)}$. Note that these scaling properties also yield the identities
\begin{equation}
  \label{eq3-9}
\begin{aligned}
B^{(2,1)}(D_n){Y}^{(1)}(\xi',0) & = \Delta_2(r) B^{(2,1)}(D_n){Y}^{(1)}(\tfrac{\xi'}r,0)\Delta_1(r)^{-1},\\
  B^{(1,2)}(D_n){Y}^{(2)}(\xi',0,q) & =  \Delta_1(r) B^{(1,2)}(D_n){Y}^{(2)} (\tfrac{\xi'}r, 0,\tfrac q r)\Delta_2(r)^{-1}.
\end{aligned}
\end{equation}
\end{remark}
%%%%%%%%%%%%%%%%%%%%%%%%%%%%%%%%%%%%%%%%%%%%%%%%%%%%%%%%%%%%%%%%%%%%%%%%%%%%%%%%%%%%%%%%%%%%%%%%%%%%%%%%%%%%%%%%%%%%%%%%%%%%%%%%%%%%
We summarize the representation of the solution in form of solution operators:

\begin{lemma}
  \label{3.5} Let $g\in\prod_{j=1}^{2m} W_p^{2m-j+1-1/p}(\R^{n-1})$, and let $u\in W_p^{2m}(\R^n_+)$ be a solution of \eqref{eq3-1}. Let $\tilde g\in \prod_{j=1}^{2m} W_p^{2m-j+1}(\R^n_+)$ be an extension of $g$ to the half-space. Then $u$ has the form
  \[ u = T_1\tilde g + T_2(\partial_n \tilde g),\]
  where $\partial_n := \frac{\partial}{\partial x_n}$ and where the solution operators $T_1$ and $T_2$ are given by
  \begin{align*}
    (T_1\varphi)(x',x_n) & = -\int_0^\infty (\mathcal F')^{-1}(\partial_n {Y})(\xi',x_n+y_n,q)\Psi(\xi',q)(\mathcal F'\varphi)(\xi',y_n)dy_n,\\
    (T_2\varphi)(x',x_n) & = -\int_0^\infty (\mathcal F')^{-1} {Y}(\xi',x_n+y_n,q)\Psi(\xi',q)(\mathcal F'\varphi)(\xi',y_n)dy_n.
  \end{align*}
  Here the basic solution ${Y}(\xi',x_n,q)$ is defined in Definition~\ref{basicfunctions}, and the coefficient matrix $\Psi(\xi',q)$ is defined in \eqref{eq3-4}.
\end{lemma}

\begin{proof}
  By definition of the fundamental solution, we have $u=(\mathcal F')^{-1}\omega(\cdot,x_n)\mathcal Fg$. Writing this in the form
  \[ u = -\int_0^\infty \tfrac{\partial}{\partial {y_n}}\big[ (\mathcal F')^{-1} \omega(\cdot,x_n+y_n)(\mathcal F'\tilde g)(\cdot,y_n)\big]dy_n, \]
  (``Volevich trick'') and noting that $\omega(\xi',x_n,q)={Y}(\xi',x_n,q)\Psi(\xi',q)$, we obtain the above representation.
\end{proof}

 Our proofs are based on the Fourier multiplier concept, see, e.g., \cite{DHP}. Here a function $m\in L^\infty(\R^n)$ is called an $L^p$-Fourier multiplier if $T_m\colon \mathscr S(\R^n)\to L^\infty(\R^n)$, $f\mapsto \mathcal F^{-1} m\mathcal F$  (being defined on the Schwartz space $\mathscr S(\R^n)$) extends to a continuous mapping $T_m\in L(L^p(\R^n))$. We will apply Michlin's theorem to prove the Fourier multiplier property. For this, we introduce the notion of a Michlin function.

 \begin{definition}
   \label{3.6} Let $M\colon (\R^{n-1}\setminus \{0\})\times \bar\Sigma\to \C^{k\times \ell}$ be a matrix-valued function. Then we call $M$ a Michlin function if $M(\cdot, q)\in C^{[\frac n2]+1}(\R^{n-1}\setminus\{0\})$ for all $q\in\bar\Sigma$ and if there exists a constant $C>0$, independent of $q$, $\gamma'$, and $\xi'$, such that
   \[ |\xi'|^{|\gamma'|}\Big| \partial_{\xi'}^{\gamma'} M(\xi',q)\Big| \le C \quad (\xi'\in\R^{n-1}\setminus\{0\},\, q\in\bar\Sigma,\, \gamma'\in\N_0^{n-1}\text{ with }|\gamma'|\le [\tfrac n2]+1).\]
 \end{definition}

 \begin{remark}
   \label{3.7}
   a) Michlin's theorem (see \cite{ITFSDO}, Section 2.2.4) states that every Michlin function is an $L^p$-Fourier multiplier for all $p\in (1,\infty)$.

   b) By the product rule one immediately sees that the product of Michlin functions is a Michlin function, too.

   c) Let $M\colon (\R^{n-1}\setminus\{0\})\times \bar\Sigma\to \C^{k\times k}$ be a Michlin function, and let $M(\xi',q)$ be invertible for all $\xi'$ and $q$. If the norm of the inverse matrix is bounded by a constant independent of $\xi'$ and $q$, then also $(\xi',q)\mapsto M(\xi',q)^{-1}$ is a Michlin function. This follows iteratively noting that
   \[\xi_j\partial_{\xi_j}M(\xi',q)^{-1} = M(\xi',q)^{-1} \Big( \xi_j\partial_{\xi_j} M(\xi',q)\Big) M(\xi',q)^{-1}.\]
    \end{remark}

Now we will show that the basic solution ${Y}$ as well as the coefficient matrix $\Psi$ satisfy uniform estimates.
Here and in the following, $C$ stands for a generic constant which may vary from inequality to inequality but is independent of the variables appearing in the inequality. We will scale the functions with $|\xi'|$ and with
\begin{equation}
  \label{eq3-5}
  \rho := \rho(\xi',q) := |\xi'|+|q|.
\end{equation}

\begin{lemma}
  \label{3.8}
  a)  For all $\ell\in\N_0$ and all $x_n>0$, the function
  \[M_1^{(\ell)}(\xi',x_n,q) := x_n\begin{pmatrix}
    |\xi'|^{-\ell} & 0\\ 0 &\rho^{-\ell}
  \end{pmatrix}
  \partial_n^{\ell+1}{Y}(\xi',x_n,q)\begin{pmatrix} \Delta_1(|\xi'|) & 0 \\ 0 & \Delta_2(\rho)
  \end{pmatrix}\]
  is a Michlin function with constant independent of $x_n\in (0,\infty)$.

  b) The functions
  \begin{align*}
    C_1(\xi',q) & := \Delta_1(\rho)^{-1} \Big(B^{(1,2)}(D_n){Y}^{(2)}\Big)(\xi',0,q) \; \Delta_2(\rho),\\
    C_2(\xi') & := \Delta_2(|\xi'|)^{-1} \Big( B^{(2,1)}(D_n){Y}^{(1)}\Big) (\xi',0)  \; \Delta_1(|\xi'|)
  \end{align*}
  are Michlin functions.
\end{lemma}

\begin{proof}
We use an explicit description of the basic solutions. According to \cite{ADN59}, Section~1, there exist polynomials (with respect to $\tau$)  $N_1(\xi',\tau),\ldots,N_m(\xi',\tau)$ and $ N_{m+1}(\xi',\tau,q),\ldots, N_{2m}(\xi',\tau,q) $ such that
\begin{eqnarray*}
\frac{1}{2\pi i} \int_{\gamma_1}\frac{N_k(\xi',\tau)}{A_{1+}(\xi',\tau)}\tau^{j-1}d\tau &=& \de_{jk}\quad (j,k=1,\ldots, m),\\ \frac{1}{2\pi i} \int_{\gamma_2}\frac{N_k(\xi',\tau,q)}{A_{2+}(\xi',\tau,q)}\tau^{j-1}d\tau &=& \de_{jk}\quad (j,k=m+1,\ldots,2m)
\end{eqnarray*}
with $\delta_{jk}$ being the Kronecker delta symbol.
Here $\gamma_1=\gamma_1(\xi')$ is a smooth closed contour in the upper half-plane  $\C_+$, depending on $\xi'$ and enclosing the $m$ roots of the polynomial $A_1(\xi',\cdot)$ with positive imaginary part, while $\gamma_2=\gamma_2(\xi',q)$ is a smooth closed contour in $\C_+$ depending on $(\xi',q)$ and enclosing the $m$ roots of $A_2(\xi',\cdot,q)$ in $\C_+$.
Moreover, $N_k$ is positively homogeneous in its arguments of degree $m-k$ for $k=1,\dots,2m$ while $A_{1+}$ and $A_{2+}$ are positively homogeneous in their arguments of degree $m$.

This leads to the following representation for the basic solutions ${Y}^{(1)} = ({Y}_k^{(1)})_{k=1,\dots,m}$ and ${Y}^{(2)} = ({Y}^{(2)}_k)_{k=m+1,\dots,2m}$:
\begin{equation}\label{eq3-8}
\begin{aligned}
  {Y}_k^{(1)} & = \frac1{2\pi i} \int_{\gamma_1} \frac{N_k(\xi',\tau)}{A_{1+}(\xi',\tau)}e^{ix_n\tau}d\tau\quad (k=1,\dots,m),\\
  {Y}_k^{(2)} & = \frac1{2\pi i} \int_{\gamma_2} \frac{N_k(\xi',\tau,q)}{A_{2+}(\xi',\tau,q)}e^{ix_n\tau}d\tau\quad (k=m+1,\dots,2m).
\end{aligned}
\end{equation}
To prove part a), we will show that for all $j\in\N_0$
\begin{equation}
  \label{eq3-6}
  x_n|\xi'|^{k-j} \partial_n^j{Y}^{(1)}_k(\xi',x_n)\;\text{ and }\; x_n \rho^{k-j}\partial_n^j {Y}_k^{(2)}(\xi',x_n,q)
\end{equation}
are Michlin functions. Setting $j:=\ell+1$ and noting the definitions of $\Delta_1$ and $\Delta_2$, this immediately implies a). Similarly, to show b) we have to prove that
\begin{equation}
  \label{eq3-7}
 |\xi'|^{k-j-1} \partial_n^j{Y}^{(1)}_k(\xi',0)\;\text{ and } \;\rho^{k-j-1}\partial_n^j {Y}_k^{(2)}(\xi',0,q)
\end{equation}
are Michlin functions. We will restrict ourselves to ${Y}_k^{(2)}$, the result for ${Y}_k^{(1)} $ follows in the same way.

For $j\in\N_0$ and $k\in\{m+1,\dots,2m\}$, we substitute $\tau\mapsto \tau/\rho$ in the integral representation \eqref{eq3-8} and obtain
\begin{align*}
  \partial_{\xi'}^{\gamma'} \big[x_n&\rho^{k-j} \partial_n^j{Y}_k^{(2)}(\xi',x_n,q)\big] \\
  &= \frac1{2\pi i}\int_{\gamma_2(\xi',q)} \partial_{\xi'}^{\gamma'}\Big[ \rho^{k-j} \frac{N_k(\xi',\tau,q)}{A_{2+}(\xi',\tau,q)}\Big] \tau^j x_n e^{ix_n\tau}d\tau \\
  & = \frac1{2\pi i} \int_{\gamma_2(\xi'/\rho, q/\rho)} \partial_{\xi'}^{\gamma'}\Big[ \rho^{k-j} \frac{N_k(\xi', \rho\tau,q)}{A_{2+}(\xi',\rho\tau,q)}\Big] (\rho\tau)^j x_ne^{i\rho x_n \tau} \rho d\tau  \\
   & = \frac1{2\pi i} \int_{\tilde\gamma_2} \rho^{j} \partial_{\xi'}^{\gamma'} \big[H_k(\xi',\rho\tau,q) \big] \tau^j (\rho x_n)e^{i\rho x_n \tau}  d\tau
\end{align*}
with $H_k(\xi',\rho,\tau) := \rho^{k-j} N_k(\xi', \rho\tau,q) / A_{2+}(\xi',\rho\tau,q)$. Note for the first equality that it is not necessary to differentiate the contour $\gamma_2(\xi',q)$ because it may be chosen locally independent of $\xi'$. In the last equality, we replaced the contour $\gamma_2(\frac{\xi'}\rho, \frac q\rho)$ by a fixed contour $\tilde \gamma_2$ which is possible by a compactness argument.

Due to the properties of $N_k$ and $A_{2+}$, the function $H_k$ is homogeneous of degree $-j$ in its arguments. Therefore, $\partial_{\xi'}^{\gamma'}H_k$ is homogeneous of degree $-j-|\gamma'|$ in its arguments, and we obtain
\[ \partial_{\xi'}^{\gamma'} \big[ H_k(\xi',\rho\tau,q) \big]= \rho^{-j-|\gamma'|} \Big(\partial_{\xi'}^{\gamma'} H_k\Big) (\tfrac{\xi'}\rho, \tau, \tfrac q\rho).\]
From the fact that $\tilde \gamma_2$ may be chosen in $\C_+$ and  the elementary inequality $te^{-t}\le 1$ $(t \ge 0)$ we get
\[ \big| (\rho x_n)e^{i\rho x_n\tau}\big| = (\rho x_n) e^{- \rho x_n\Im\tau} \le \frac{1}{\Im \tau} \le C\]
for $\tau\in\tilde\gamma_2$. Inserting this and the homogeneity of $H_k$ into the above representation, we see
\[ \Big| \partial_{\xi'}^{\gamma'} \big[x_n\rho^{k-j} \partial_n^j{Y}_k^{(2)}(\xi',x_n,q)\big] \Big| \le C \rho^j \rho^{-j-|\gamma'|} \le C |\xi'|^{-|\gamma'|}\]
which shows \eqref{eq3-6}. In the same way, for the proof of \eqref{eq3-7} we set $x_n=0$ in the above integral representation and obtain
\begin{align*}
 \Big|\partial_{\xi'}^{\gamma'} \big[&\rho^{k-j-1} \partial_n^j{Y}_k^{(2)}(\xi',0,q)\big] \Big|  = \Big|\frac1{2\pi i} \int_{\tilde\gamma_2} \rho^{j} \partial_{\xi'}^{\gamma'} \big[H_k(\xi',\rho\tau,q) \big] \tau^j   d\tau\Big|\\
 & \le C \rho^j \rho^{-j-|\gamma'|} \le C |\xi'|^{-|\gamma'|}.
\end{align*}
This finishes the proof of \eqref{eq3-6} and \eqref{eq3-7} for ${Y}^{(2)}_k$. For ${Y}_k^{(1)}$, we use the substitution $\tau\mapsto \tau/{|\xi'|}$ in the integral representation. As indicated above, a) and b) are immediate consequences of \eqref{eq3-6} and \eqref{eq3-7}, respectively.
\end{proof}

The last lemma in connection with the following result is the essential step for the proof of the a priori estimates from the main theorem.

\begin{lemma}
  \label{3.9}
  The functions
  \begin{align*}
   M_2(\xi',q) & :=  \begin{pmatrix}
    \Delta_1(|\xi'|)^{-1} & 0 \\ 0 & |\xi'|^{-m}\Delta_1(\rho)^{-1}
  \end{pmatrix} \Psi(\xi',q)
  \begin{pmatrix}
    \Delta_1(|\xi'|) & 0 \\ 0 & |\xi'|^{m}\Delta_1(\rho)
  \end{pmatrix},\\[1em]
  \tilde M_2(\xi',q)     &:= \begin{pmatrix}
    |\xi'|^{-1} I_m & 0 \\ 0 & \rho^{-1} I_m
  \end{pmatrix} M_2(\xi',q)
  \begin{pmatrix}
    |\xi'| I_m & 0 \\ 0 & \rho I_m
  \end{pmatrix}
  \end{align*}
  are Michlin functions.
\end{lemma}

\begin{proof}
By Lemma~\ref{3.8} b), we have
\[ \Psi(\xi',q) = \begin{pmatrix}
  I_m & \Delta_1(\rho) C_1 (\xi',q) \Delta_2(\rho)^{-1}\\
  \Delta_2(|\xi'|) C_2(\xi') \Delta_1(|\xi'|)^{-1} & I_m
\end{pmatrix}^{-1}\]
with Michlin functions $C_1$ and $C_2$. For $M_2$ we obtain
\begin{equation}
\label{eq3-10}
 M_2(\xi',q) =
\begin{pmatrix}
  I_m & \big(\frac{|\xi'|}{\rho}\big)^m\Delta_1\big(\frac{\rho}{|\xi'|}\big) C_1(\xi',q)\\
  \Delta_1\big(\frac{|\xi'|}{\rho}\big) C_2(\xi')  & I_m
\end{pmatrix}^{-1}.
\end{equation}
By a homogeneity argument we see that $\Delta_1(|\xi'|/\rho)$ and $(|\xi'|/\rho)^m \Delta_1(\rho/|\xi'|)$ are Michlin functions, and therefore the matrix on the right-hand side of \eqref{eq3-10} is a Michlin function. In order to  apply Remark~\ref{3.7} c), we have to show that the norm of $M_2(\xi',q)$ is uniformly bounded.

For this, we write $M_2(\xi',q)$ in the form of a Schur complement: For an invertible block matrix, we have
\[ \begin{pmatrix} I_m & A^{(1,2)} \\ A^{(2,1)} & I_m\end{pmatrix}^{-1}
= \begin{pmatrix}
  I_m + A^{(1,2)} S^{-1} A^{(2,1)} & - A^{(1,2)} S^{-1}\\
  - S^{-1} A^{(2,1)} & S^{-1}
\end{pmatrix}\]
with $S:= I_m-A^{(2,1)} A^{(1,2)}$. Applied to the matrix $M_2$, we obtain
\begin{equation}
  \label{eq3-12}
M_2(\xi',q) = \begin{pmatrix}
  I_m + \big(\frac{|\xi'|}{\rho}\big)^m\Delta_1\big(\frac{\rho}{|\xi'|}\big) C_1 S^{-1}\Delta_1\big(\frac{|\xi'|}{\rho}\big) C_2 &
  - \big(\frac{|\xi'|}{\rho}\big)^m\Delta_1\big(\frac{\rho}{|\xi'|}\big) C_1 S^{-1}\\
  - S^{-1} \Delta_1\big(\frac{|\xi'|}{\rho}\big) C_2 & S^{-1}
\end{pmatrix}
\end{equation}
with
\begin{equation}\label{eq3-11}
 S(\xi',q) := I_m - \big(\tfrac{|\xi'|}{\rho}\big)^m \Delta_1\big(\tfrac{|\xi'|}{\rho}\big) C_2(\xi') \Delta_1\big(\tfrac{\rho}{|\xi'|}\big) C_1(\xi',q).
\end{equation}
By \eqref{eq3-9}, the matrices $C_1$ and $C_2$ and, consequently, the matrix $M_2$ are homogeneous of degree 0 in their arguments. Thus we can write $S$ in the form
\[ S(\xi',q) = S\Big(\frac {\xi'}{|\xi'|}, \frac{q}{|\xi'|}\Big)\quad (\xi'\in\R^{n-1}\setminus\{0\},\, q\in\bar\Sigma).\]
We set $\eta':=\xi'/|\xi'|$ and
\[ \Lambda:= \frac{\rho}{|\xi'|} = \frac{|\xi'|+ |q| } {|\xi'|}= 1+\frac{|q|}{|\xi'|} \]
and write $S$ as
\[ S(\xi',q) = I_m - \Lambda^{-m} \Delta_1\Big( \frac{1}{\Lambda}\Big)C_2(\eta')
\Delta_1(\Lambda) C_1\Big(\eta',\frac q{|\xi'|}\Big). \]
The matrices $C_1, C_2$, and $\Delta_1(1/\Lambda)$ are bounded for all $\xi'\in\R^{n-1}\setminus\{0\}$ and $q\in\bar\Sigma$.
By $|\Delta_1(\Lambda)|\le C \Lambda^{m-1}$ for all $\Lambda \geq 1$, we see that there exists a $\Lambda_0>1$ such that
\[ \Big|  \Lambda^{-m} \Delta_1\Big( \frac{1}{\Lambda}\Big)C_2(\eta')
\Delta_1(\Lambda) C_1\Big(\eta',\frac q{|\xi'|}\Big) \Big| \le \frac 12  \]
holds for all $\xi'\in\R^{n-1}\setminus\{0\}$ and $q\in\bar\Sigma$ with $|q|\ge \Lambda_0\,|\xi'|$. For these $\xi'$ and $q$, a Neumann series argument shows that the norm of $S^{-1}(\xi',q)$ is bounded by 2.

For $\xi'\in\R^{n-1}\setminus\{0\}$ and $q\in\bar\Sigma$ with $|q|\le \Lambda_0|\xi'|$, the tuple $(\xi'/|\xi'|, q/|\xi'|)$ belongs to the compact set $\big\{ (\eta',\tilde q): |\eta'|=1, \, \tilde q\in \bar\Sigma,\, |\tilde q|\le \Lambda_0\big\}$. Now we use the fact that for all $\xi'\in\R^{n-1}\setminus\{0\}$ and $q\in\bar\Sigma$, the matrix $B(D_n){Y}(\xi',0,q)$ is invertible, and therefore  the matrix on the right-hand side of \eqref{eq3-10} is invertible, too. This yields the invertibility of $S$, and by continuity  the inverse matrix $S^{-1}(\xi',q)$ is bounded   for these $\xi'$ and $q$.

Therefore, we have seen that $|S^{-1}(\xi',q)|\le C$ holds for all $\xi'\in\R^{n-1}\setminus\{0\}$ and $q\in\bar\Sigma$. From the explicit description of $M_2(\xi',q)$ in \eqref{eq3-12} and the uniform boundedness of the other coefficients in \eqref{eq3-12}, we see that $|M_2(\xi',q)|\le C$ holds for all $\xi'$ and $q$. By Remark~\ref{3.7} c), $M_2$ is a Michlin function.

The above proof also shows that the modification $\tilde M_2$ is a Michlin function. Note that $S(\xi',q)$ remains unchanged and that we obtain an additional factor $\rho/|\xi'|$ in the right upper corner which does not affect the boundedness.
 \end{proof}

\section{Proof of the a priori estimate}

%%%%%%%%%%%%%%%%%%%%%%%%%%%%%%%%%%%%%%%%%%%%%%%%%%%%%%%%%%%%%%%%%%%%%%%%%%%%%%%%%%%%%%%%%%%%%%%%%%%%%%%%%%%%%%%

In this section, we will investigate the mapping properties of the solution operators $T_1, T_2$ introduced in Lemma~\ref{3.5}. As above, let $\rho := |\xi'|+|q|$. In the following, we will use the abbreviations $D':= -i(\frac{\partial}{\partial x_1}\cdots\frac\partial{\partial x_{n-1}})$ and  $L(D',q):= (\mathcal F')^{-1} L(\xi',q) \mathcal F'$. Based on Lemma~\ref{3.8} and \ref{3.9} and on the continuity of the Hilbert transform, it is not difficult to obtain the following result.

\begin{lemma}
  \label{4.1}
  a) Let
  \[ L_1(\xi',q) := \begin{pmatrix}
    \rho^m|\xi'|^m & 0 \\ 0 & \rho^{2m}
  \end{pmatrix},\quad L_2(\xi',q) :=
  \begin{pmatrix}
    \rho^m |\xi'|^m \Delta_1(|\xi'|)^{-1} & 0 \\ 0 & \rho^m\Delta_1(\rho)^{-1}
  \end{pmatrix}.
  \]
  Then for all $\varphi\in\mathscr S(\R^n_+)^{2m}$ and all $\ell\in\N_0$ we have
  \[ \left\| L_1(D',q) \begin{pmatrix}
    |D'|^{-\ell} & 0 \\ 0 & (|D'|+|q|)^{-\ell}
  \end{pmatrix}\partial_n^\ell T_1\varphi \right\|_{L^p(\R^n_+)} \le C \| L_2(D',q) \varphi\|_{L^p(\R^n_+)}.\]
  The same holds when $L_1$ and $L_2$ are replaced by $L_1^{(0)}:=|\xi'|^m\rho^{-m}L_1$ and $L_2^{(0)}:=|\xi'|^m\rho^{-m} L_2$, respectively.

  b) Let
  \[ \tilde L_2(\xi',q) :=
  \begin{pmatrix}
    \rho^m |\xi'|^{m-1} \Delta_1(|\xi'|)^{-1} & 0 \\ 0 & \rho^{m-1}\Delta_1(\rho)^{-1}
  \end{pmatrix}. \]
  Then for all $\varphi\in\mathscr S(\R^n_+)^{2m}$ and all $\ell\in\N_0$ we have
  \[ \left\| L_1(D',q) \begin{pmatrix}
    |D'|^{-\ell} & 0 \\ 0 & (|D'|+|q|)^{-\ell}
  \end{pmatrix}\partial_n^\ell T_2\varphi \right\|_{L^p(\R^n_+)} \le C \| \tilde L_2(D',q) \varphi\|_{L^p(\R^n_+)}.\]
  \end{lemma}

\begin{proof}
a) For fixed $\ell\in\N_0$, let $\tilde \varphi:= L_2(D',q)\varphi$ and
\[\tilde u := L_1(D',q) \begin{pmatrix}
    |D'|^{-\ell} & 0 \\ 0 & (|D'|+|q|)^{-\ell}
  \end{pmatrix}\partial_n^\ell T_1\varphi.\]
We have to show that $\|\tilde u\|_{L^p(\R^n_+)}\le C \|\tilde \varphi\|_{L^p(\R^n_+)}$. For this, we write
\begin{align*}
  L_1(\xi',q) & \begin{pmatrix}
    |\xi'|^{-\ell} & 0 \\ 0 & \rho^{-\ell}
  \end{pmatrix} \partial_n^{\ell+1}{Y}(\xi',x_n,q)\Psi(\xi',q) L_2(\xi',q)^{-1} \\
  & = x_n^{-1} M_1^{(\ell)}(\xi',x_n,q)\Big( \rho^m|\xi'|^m I_{2m}\Big) M_2(\xi',q) \Big( \rho^{-m}|\xi'|^{-m}I_{2m}\Big)\\
  &= x_n^{-1} M_1^{(\ell)}(\xi',x_n,q) M_2(\xi',q).
\end{align*}
Inserting this into the definition of the solution operator, we obtain
\[\tilde u = -\int_0^\infty \frac1{x_n+y_n} (\mathcal F')^{-1} M_1^{(\ell)}(\xi',x_n+y_n,q)M_2(\xi',q) \mathcal F'\tilde \varphi(\xi',y_n)\,dy_n.\]
Therefore,
\begin{align*}
  \|\tilde u\|_{L^p(\R^n_+)}^p & = \int_0^\infty \|\mathcal F'\tilde u(\cdot,x_n)\|_{L^p(\R^{n-1})}^p dx_n\\
  & \le \int_0^\infty \Bigg[ \int_0^\infty \frac1{x_n+y_n}\Big\| (\mathcal F')^{-1}M_1^ {(\ell)}(\xi',x_n+y_n,q)\\
  & \hspace*{8em}M_2(\xi',q) \mathcal F'\tilde \varphi(\xi',y_n)\Big\|_{L^p(\R^{n-1})} dy_n\Bigg]^p dx_n\\
  & \le C \int_0^\infty \Bigg[ \int_0^\infty \frac1{x_n+y_n}\|\tilde \varphi(\cdot,x_n)\|_{L^p(\R^{n-1})} dy_n\Bigg]^p dx_n\\
  & \le C \int_0^\infty \|\tilde \varphi(\cdot,y_n)\|_{L^p(\R^{n-1})}^p dy_n\\
  & = C\|\tilde \varphi\|^p_{L^p(\R^n_+)}.
\end{align*}
Here we used the fact that $M_1^{(\ell)}$ and $M_2$ are Michlin functions and therefore Fourier multipliers and that the (one-sided) Hilbert transform
\[\phi\mapsto H\phi,\; (H\phi)(x_n) := \int_0^\infty \frac{\phi(y_n)}{x_n+y_n}\, dy_n\]
induces a bounded operator in $L^p((0,\infty))$ for every $p\in (1,\infty)$.

This shows the first statement in a). Obviously, the uniform estimate also holds in the case when $L_1$ and $L_2$ are multiplied with the same factor, as this factor cancels out.

The proof of b) follows exactly in the same way with $M_2$ being replaced by $\tilde M_2$ from Lemma~\ref{3.5}.
\end{proof}

The next result shows the key estimate for the solution of \eqref{eq3-1}.

\begin{theorem}
  \label{4.2}
  Let $u\in (W_p^{2m}(\R^n_+))^2$ be a solution of $A(D,q)u = 0,\, B(D_n) u = g$ with $g\in \prod_{j=1}^{2m} W_p^{2m-j+1-1/p}(\R^{n-1})$. Let $\tilde g\in \prod_{j=1}^{2m} W_p^{2m-j+1}(\R^n_+)$ be an extension of $g$ to the half-space. Let $q_0>0$. Then for all $q\in\bar\Sigma$ with $|q|\ge q_0$, the inequalities
  \begin{align*}
  |u_1|_{2m,p,\R^n_+} + |u_2|_{m,p,\R^n_+} + |q|^m|u_2|_{0,p,\R^n_+}& \le C \sum_{j=1}^{2m} \|\tilde g_j\|_{2m-j+1,p,\R^n_+},\\
    |u_1|_{2m,p,\R^n_+}  + |q|^m |u_1|_{m,p,\R^n_+} + \norm u_2\norm_{2m,p,\R^n_+} & \le C\sum_{j=1}^{2m} \norm \tilde g_j\norm_{2m-j+1,p,\R^n_+}
  \end{align*}
  hold.
\end{theorem}

\begin{proof}
  In this proof, we will write $\|\cdot\| := \|\cdot\|_{L^p(\R^n_+)}$. We use the equivalences
  \[
   |\varphi|_{k,p,\R^n_+}  \approx \sum_{\ell=0}^k \big\| |D'|^{k-\ell}\partial_n^\ell \varphi\big\|,\quad
\norm\varphi\norm_{k,p,\R^n_+}  \approx \sum_{\ell=0}^k \big\| (|D'|+|q|)^{k-\ell}\partial_n^\ell\varphi\big\|
\]
which can easily be seen by a Michlin type argument. With this, we get
\begin{align*}
  |u_1|_{2m,p,\R^n_+} &+ |q|^m |u_1|_{m,p,\R^n_+} + \norm u_2\norm_{2m,p,\R^n_+}\\
  & \approx \sum_{\ell=0}^{2m} \big\| |D'|^{2m-\ell} \partial_n^\ell u_1 \big\|  + \sum_{\ell=0}^m |q|^m \big\| |D'|^{m-\ell} \partial_n^\ell u_1\big\| \\
  &\quad +  \sum_{\ell=0}^{2m} \big\| (|D'|+|q|)^{2m-\ell}\partial_n^\ell u_2\big\|\\
  & \le C \sum_{\ell=0}^{2m} \Bigg\| L_1(D',q)\begin{pmatrix}|D'|^{-\ell} & 0 \\ 0 & (|D'|+|q|)^{-\ell}\end{pmatrix} \partial_n^\ell u\Bigg\|
\end{align*}
and
\begin{align*}
  \|L_2(D',q) \tilde g\| & \approx \sum_{j=1}^m \big\| (|D'|+|q|)^m |D'|^{m-j+1} \tilde g_j\big\|  \\
  & \qquad  + \sum_{j=m+1}^{2m} \big\| (|D'|+|q|)^{2m-j+1} \tilde g_j\big\|\\
  &\approx \sum_{j=1}^m |\tilde g_j|_{2m-j+1,p,\R^n_+} + |q|^m \sum_{j=1}^m |\tilde g_j|_{m-j+1,p,\R^n_+} \\
  & \qquad + \sum_{j=m+1}^{2m}\norm \tilde g_j\norm_{2m-j+1,p,\R^n_+}.
\end{align*}
Let $u\in (W_p^{2m}(\R^n_+))^2$ be a solution of $A(D,q)u=0,\, B(D_n)u=g$, and let $\tilde g$ be an extension of $g$. By a density argument, we may assume that $u\in (\mathscr S(\R^n_+))^2$. By Lemma~\ref{3.5}, we have $u=T_1\tilde g+T_2(\partial_n \tilde g)$. Applying Lemma~\ref{4.1}, we get
\begin{align}
  |u_1&|_{2m,p,\R^n_+}+ |q|^m |u_1|_{m,p,\R^n_+} + \norm u_2\norm_{2m,p,\R^n_+}\nonumber \\
  &\le C \Bigg[ \sum_{j=1}^m \left(|\tilde g_j|_{2m-j+1,p,\R^n_+} + |q|^m  |\tilde g_j|_{m-j+1,p,\R^n_+} \right) \\
  & \quad + \sum_{j=m+1}^{2m}\norm \tilde g_j\norm_{2m-j+1,p,\R^n_+}\nonumber\\
  &\quad + \sum_{j=1}^m |\partial_n \tilde g_j|_{2m-j,p,\R^n_+} + |q|^m |\partial_n \tilde g_j|_{m-j,p,\R^n_+} + \sum_{j=m+1}^{2m} \norm \partial_n \tilde g_j\norm_{2m-j,p,\R^n_+}\Bigg] \nonumber \\
  & \le C
  \Big[ \sum_{j=1}^m \left(\|\tilde g_j\|_{2m-j+1,p,\R^n_+} + |q|^m \|\tilde g_j\|_{m-j+1,p,\R^n_+} \right) \\
  & \quad+ \sum_{j=m+1}^{2m}\norm \tilde g_j\norm_{2m-j+1,p,\R^n_+}\Big].
  \label{eq4-1}
\end{align}
Inserting the inequality
\[ \|\tilde g_j\|_{2m-j+1,p,\R^n_+} +|q|^m \|\tilde g_j\|_{m-j+1,p,\R^n_+} \le C \norm \tilde g_j\norm_{2m-j+1,p,\R^n_+}\]
into the right-hand side, we see that
\begin{equation}
   |u_1|_{2m,p,\R^n_+}+ |q|^m |u_1|_{m,p,\R^n_+} + \norm u_2\norm_{2m,p,\R^n_+}
   \le C\sum_{j=1}^{2m} \norm \tilde g_j\norm_{2m-j+1,p,\R^n_+}.
   \label{eq4-2}
\end{equation}
On the other hand, inserting the inequality
\[ \|\tilde g_j\|_{2m-j+1,p,\R^n_+} +|q|^m \|\tilde g_j\|_{m-j+1,p,\R^n_+} \le C |q|^m \| \tilde g_j\|_{2m-j+1,p,\R^n_+}\]
(which holds for all $|q|\ge q_0$ with a constant $C$ depending on $q_0$), we get in particular
\[ \norm u_2\norm_{2m,p,\R^n_+} \le C |q|^m \sum_{j=1}^{2m} \| \tilde g_j\|_{2m-j+1,p,\R^n_+}.\]
Dividing by $|q|^m$, we see that this implies
\begin{equation}
  \label{eq4-3}
  |u_2|_{m,p,\R^n_+} + |q|^m \|u_2\| \le C \sum_{j=1}^{2m} \|\tilde g_j\|_{2m-j+1,p,\R^n_+}.
\end{equation}
In the same way as above, we can apply Lemma~\ref{4.1} with $L_1^{(0)}$ and $L_2^{(0)}$ instead of $L_1$ and $L_2$, respectively. We see that
\begin{align*}
|u_1&|_{2m,p,\R^n_+} \approx \sum_{\ell=0}^{2m} \big\| |D'|^{2m-\ell} \partial_n^\ell u_1\big\|  \\
& \le C \sum_{\ell=0}^{2m} \Big\| L_1^{(0)} (D',q) \begin{pmatrix}
  |D'|^{-\ell} &  0\\ 0 & (|D'|+|q|)^{-\ell}
\end{pmatrix} \partial_n^\ell u\Big\| \\
& \le C \Big( \| L_2^{(0)} (D',q) \tilde g\| + \| \tilde L_2^{(0)} (D',q) \partial_n \tilde g\|\Big)\\
& \le C \Big( \sum_{j=1}^m \big\| |D'|^{2m-j+1} \tilde g_j\big\| + \sum_{j=m+1}^{2m} \big\| |D'|^m (|D'|+|q|)^{m-j+1} \tilde g_j\big\| \\
&\quad + \sum_{j=1}^m \big\| |D'|^{2m-j} \partial_n \tilde g_j\big\| + \sum_{j=m+1}^{2m} \big\| |D'|^m (|D'|+|q|)^{m-j} \partial_n \tilde g_j\big\|\Big) .
\end{align*}
With the inequality
\[  \big\| |D'|^m (|D'|+|q|)^{m-j+1} \tilde g_j\big\|  \le C \big\| |D'|^{2m-j+1} \tilde g_j\big\|\quad (j=m+1,\dots,2m)\]
this gives
\[ |u_1|_{2m,p,\R^n_+} \le C \sum_{j=1}^{2m} \| \tilde g_j\|_{2m-j+1,p,\R^n_+}.\]
This and equations \eqref{eq4-2} and \eqref{eq4-3} yield the statements of the theorem.
\end{proof}

Now we can consider the problem $A(D,q)u=f, B(D_n)u =g$ in the half-space. As mentioned in Remark~\ref{2.3}, this finishes the proof of the main theorem.

\begin{theorem}
  \label{4.3}
  Let $u\in (W_p^{2m}(\R^n_+))^2$ be a solution of $A(D,q)u = f,\, B(D_n) u = g$ with $f\in (L^p(\R^n_+))^2$ and $g\in \prod_{j=1}^{2m} W_p^{2m-j+1-1/p}(\R^{n-1})$.  Let $q_0>0$. Then for all $q\in\bar\Sigma$  with $|q|\ge q_0$ the following a priori estimates hold:
  \begin{align}
  \|u_1\|_{2m,p,\R^n_+} &+ \|u_2\|_{m,p,\R^n_+} + |q|^m\|u_2\|_{0,p,\R^n_+} \le C  \Big( \|f_1\|_{0,p,\R^n_+}\nonumber\\
  & + \|f_2\|_{0,p,\R^n_+} + \sum_{j=1}^{2m} \|g_j\|_{2m-j+1-1/p,p,\R^{n-1}}+ \|u_1\|_{0,p,\R^n_+}\Big),\label{eq4-4}\\
    \|u_1\|_{2m,p,\R^n_+}  &+ |q|^m \|u_1\|_{m,p,\R^n_+} + \norm u_2\norm_{2m,p,\R^n_+} \le C\Big( |q|^m \|f_1\|_{0,p,\R^n_+} \nonumber\\
  &  + \|f_2\|_{0,p,\R^n_+} + \sum_{j=1}^{2m} \norm g_j\norm_{2m-j+1-1/p,p,\R^{n-1}} +   |q|^m \|u_1\|_{0,p,\R^n_+}\Big).\label{eq4-5}
  \end{align}
\end{theorem}

\begin{proof}
  \textbf{(i)}  We start the proof with some preliminary remarks. Let $r_+\colon \varphi\mapsto \varphi|_{\R^n_+}$ be the restriction operator from $\R^n$ to $\R^n_+$. Then $r_+$  is a retraction from $W_p^k(\R^n)$ to $W_p^k(\R^n_+)$ for every $k\in\N_0$, and there exists a co-retraction (independent of $k$), i.e. a total extension operator  $e_+\in L(W_p^k(\R^n_+),W_p^k(\R^n))$ satisfying $r_+e_+=\id_{W_p^k(\R^n_+)}$ for all $k$ (see \cite{ADAM}, Theorem~5.21).

  For every $j\in\{0,\dots, k-1\}$, the trace operator to the boundary $\gamma_j u := \partial_n^j u|_{\R^{n-1}}$ is a bounded operator from $W_p^k(\R^n_+)$ to $W_p^{k-1/p}(\R^{n-1})$. This holds both with respect to the parameter-independent norms $\|\cdot\|$ and the parameter-dependent norms $\norm\cdot\norm$. For the latter, we refer to \cite{ADF97}, Proposition~2.2. There exists a parameter-dependent extension operator $E_q\in L(W_p^{k-1/p}(\R^{n-1}),  W_p^k(\R^n_+))$ which satisfies $\gamma_0E_q=\id_{W_p^{k-1/p}(\R^{n-1})}$ and whose operator norm with respect to the parameter-dependent norms $\norm\cdot\norm$ is bounded by a constant independent of $q$ for all $q\in\bar\Sigma$ with $|q|\ge q_0$ (see, e.g., \cite{ADF97}, Proposition~2.3). In particular, we will consider $E_1$ which is a parameter-independent continuous extension operator.

  Let $\psi\in C^\infty(\R^n)$ with $0\le \psi\le 1$, $\psi(\xi)=0$ for $|\xi|\le 1$ and $\psi(\xi)=1$ for $|\xi|\ge 2$. Then a simple application of Michlin's theorem shows that $R_1(D) := \mathcal F^{-1} \psi(\xi)A_1^{-1}(\xi)\mathcal F$ induces a bounded linear operator $R_1(D)\in L(W_p^k(\R^n), $ $W_p^{k+2m}(\R^n))$ for all $k\in\N_0$. Due to the compact support of $1-\psi$, the related operator $(1-\psi)(D)$ belongs to $L(L^p(\R^n), W_p^k(\R^n))$ for all $k\in\N_0$. Note that $A_1(D)$ and $\psi(D)$ commute due to $A_1(D)\psi(D) = \mathcal F^{-1} A_1(\xi)\psi(\xi)\mathcal F$.

\textbf{(ii)} Let $u\in (W_p^{2m}(\R^n_+))^2$ be a solution of $A(D)u=f,\, B(D)u=g$, and let $q\in\bar\Sigma$ with $|q|\ge q_0$.  We define
$ \tilde f_1 := A_1(D) e_+u_1$. Then $\tilde f_1\in L^p(\R^n)$ and $r_+\tilde f_1 = A_1(D)r_+e_+u_1 = A_1(D)u_1 = f_1$. For
\[ v_1 := r_+\Big[ (1-\psi)(D) e_+u_1 + R_1(D)\tilde f_1\Big],\]
we obtain $v_1\in W_p^{2m}(\R^n_+)$ and
\begin{align*}
  A_1(D)v_1 & = r_+ A_1(D)(1-\psi)(D)e_+u_1 + r_+ A_1(D) R_1(D)\tilde f_1\\
  & = r_+ (1-\psi)(D) A_1(D) e_+u_1 + r_+ (A_1 \psi A_1^{-1} )(D)\tilde f_1\\
  & = r_+ (1-\psi)(D)\tilde f_1 + r_+ \psi(D)\tilde f_1 = r_+\tilde f_1 = f_1.
\end{align*}
By the continuity of the involved operators, we have
\begin{equation}
  \label{eq4-6}
  \|v_1\|_{2m,p,\R^n_+} \le C\Big( \|u_1\|_{0,p,\R^n_+} + \|f_1\|_{0,p,\R^n_+}\Big).
\end{equation}

\textbf{(iii)} Similarly, we set $v_2 := r_+ A_2(D,q)^{-1} e_+ f_2$. It is well-known (or easily seen by Michlin's theorem) that $v_2\in W_p^{2m}(\R^n_+)$ with $A_2(D,q)v_2=f_2$ and
\begin{equation}
  \label{eq4-7}
  \norm v_2\norm_{2m,p,\R^n_+} \le C \|f_2\|_{0,p,\R^n_+}.
\end{equation}

\textbf{(iv)} We define $v:= (v_1,v_2)^\top\in (W_p^{2m}(\R^n_+))^2$ and $w:=u-v$. Then $w$ is a solution of $A(D)w=0,\, B(D)w=g- B(D)v$. Applying the parameter-independent extension operator $E_1$ to every component of $g$, we define $\tilde g:= E_1g\in \prod_{j=1}^{2m} W_p^{2m-j+1}(\R^n_+)$. An extension $\tilde h$ of $B(D)v$ is given by omitting the trace to the boundary. Note that $\tilde h_j = \partial_n^{j-1} v_1\pm \partial_n^{j-1} v_2$.

For the left-hand side of \eqref{eq4-4}, we remark that for $w=(w_1,w_2)^\top$ we have
\[ \|w_1\|_{2m,p,\R^n_+} \le C \big( |w_1|_{2m,p,\R^n_+} + \|w_1\|_{0,p,\R^n_+}\big).\]
By Theorem~\ref{4.3}, we obtain
\begin{align}
\|w_1\|_{2m,p,\R^n_+} & + \|w_2\|_{m,p,\R^n_+} + |q|^m \|w_2\|_{0,p,\R^n_+} \le C \Big( \|f\|_{0,p,\R^n_+}\nonumber \\
  & \quad+\sum_{j=1}^{2m} \| \tilde g_j + \partial_n^{j-1} (v_1 \pm v_2)\|_{2m-j+1,p,\R^n_+} + \|w_1\|_{0,p,\R^n_+}\Big). \label{eq4-8}
\end{align}
From \eqref{eq4-6} we see that
\[ \|\partial_n^{j-1} v_1\|_{2m-j+1,p,\R^n_+} \le C \|v_1\|_{2m,p,\R^n_+} \le C \Big( \|u_1\|_{0,p,\R^n_+} + \|f_1\|_{0,p,\R^n_+}\Big).\]
For $v_2$ we obtain $\|\partial_n^{j-1}v_2\|_{2m-j+1,p,\R^n_+}\le C\|f_2\|_{0,p,\R^n_+}$ in the same way from \eqref{eq4-7}. Inserting this into \eqref{eq4-8}, we obtain the first inequality \eqref{eq4-4} of the theorem.

\textbf{(v)} The proof of \eqref{eq4-5} follows the same lines. However, here we start with the refined estimate \eqref{eq4-1}. For the left-hand side of \eqref{eq4-5}, we note that
\[ \|u_1\|_{2m,p,\R^n_+} + |q|^m \|u_1\|_{m,p,\R^n_+} \le C\Big( |u_1|_{2m,p,\R^n_+} + |q|^m |u_1|_{m,p,\R^n_p} + |q|^m \|u_1\|_{0,p,\R^n_+}\Big).\]
Now we define $\tilde g:= E_q g$ with the parameter-dependent extension operator $E_q$ from part (i). Then the term on the right-hand side of \eqref{eq4-1} equals
\begin{align}
  & \sum_{j=1}^m \|\tilde g_j+\partial_n^{j-1}(v_1 \pm v_2)\|_{2m-j+1,p,\R^n_+} \nonumber\\
  & \qquad\qquad + |q|^m \sum_{j=1}^m \|\tilde g_j+ \partial_n^{j-1} (v_1 \pm v_2)\|_{m-j+1,p,\R^n_+}  \nonumber\\
   & \qquad\qquad + \sum_{j=m+1}^{2m}\norm \tilde g_j+ \partial_n^{j-1} (v_1 \pm v_2)\norm_{2m-j+1,p,\R^n_+}.\label{eq4-9}
\end{align}
For $j=1,\dots,m$, we can estimate
\begin{align*}
\|\tilde g_j&\|_{2m-j+1,p,\R^n_+} + |q|^m  \|\tilde g_j\|_{m-j+1,p,\R^n_+} \le C \norm \tilde g_j\norm_{2m-j+1,p,\R^n_+}\\
& \le C \norm g_j\norm_{2m-j+1-1/p,p,\R^{n-1}} .
\end{align*}
Concerning the terms involving $v_1$, we use
\begin{align*}
  \|\partial_n^{j-1}v_1&\|_{2m-j+1,p,\R^n_+} + |q|^m \|\partial_n^{j-1} v_1\|_{m-j+1,p,\R^n_+} \\
  & \le \|v_1\|_{2m,p,\R^n_+} + |q|^m \|v_1\|_{m,p,\R^n_+} \le C |q|^m \|v_1\|_{2m,p,\R^n_+}\\
  & \le C |q|^m \Big( \|u_1\|_{0,p,\R^n_+}+\|f_1\|_{0,p,\R^n_+}\Big)
\end{align*}
for $j=1,\dots,m$ and
\[ \norm \partial_n^{j-1}v_1\norm_{2m-j+1,p,\R^n_+} \le C |q|^m \|v_1\|_{2m,p,\R^n_+}\le C |q|^m \Big( \|u_1\|_{0,p,\R^n_+}+\|f_1\|_{0,p,\R^n_+}\Big)\]
for $j=m+1,\dots,2m$. Finally, the terms involving $v_2$ can be estimated by
\[ \norm \partial_n^{j-1}v_2\norm_{2m-j+1,p,\R^n_+} \le C \norm v_2\norm_{2m,p,\R^n_+}\le C \|f_2\|_{0,p,\R^n_+}.\]
So we see that all terms in \eqref{eq4-9} can be estimated by the right-hand side of \eqref{eq4-5}, and the proof of \eqref{eq4-5} is finished.
\end{proof}

%%%%%%%%%%%%%%%%%%%%%%%%%%%%%%%%%%%%%%%%%%%%%%%%%%%%%%%%%%%%%%
\begin{remark}
a) The estimate \eqref{aprioriabsch_domain} does not imply uniqueness of a solution to \eqref{TP} because the elliptic part $u_1$ of the solution appears in a norm of lower order on the right-hand side of the estimate. Nevertheless, in bounded domains such estimates give rise to the Fredholm property of a corresponding solution operator.

b) For $g=0$ and $f_1=0$, we obtain in particular
\[ |\lambda| \, \|u_2\|_{0,p,\R^n_+} \le C \Big( \|f_2\|_{0,p,\R^n_+} + |\lambda|^{1/2}\|u_1\|_{0,p,\R^n_+}\Big)\]
from \eqref{eq4-5}. This is the basis for resolvent estimates and spectral properties of the corresponding $L^p$-realization in the case where the Dirichlet problem for $A_1(x,D)$ in ${\Omega}_1$ is invertible. Here we have a connection to eigenvalue problems with weights and the Calder\'{o}n method as studied  in, e.g., \cite{Fai09}.
\end{remark}
%%%%%%%%%%%%%%%%%%%%%%%%%%%%%%%%%%%%%%%%%%%%%%%%%%%%%%%%%%%%%%%%%%%%%%%%%%%%%%%%%%%%%%%%%%%%%%%%%%%%%%%%%%%%%%%%%%%%%%%%%%%%%%%%%%%%

\end{document}